\documentclass[a4paper]{amsart}

\usepackage{amssymb, stmaryrd, ifthen, comment, mathtools, xspace}
\usepackage[alphabetic]{amsrefs}

\usepackage[margin=3.2cm]{geometry}

\DefineSimpleKey{bib}{primaryclass}{}
\DefineSimpleKey{bib}{archiveprefix}{}

\BibSpec{arXiv}{%
  +{}{\PrintAuthors}{author}
  +{,}{ \textit}{title}
  +{}{ \parenthesize}{date}
  +{,}{ arXiv }{eprint}
  +{,}{ primary class }{primaryclass}
}

\mathtoolsset{centercolon}

\usepackage[shortlabels]{enumitem}
\setenumerate[0]{label={\rm (\roman*)}, leftmargin=*}

\usepackage[T1]{fontenc}
\usepackage[english]{babel}

\usepackage{hyperref}\hypersetup{bookmarksopen=true}

\usepackage[textsize=footnotesize,textwidth=20ex,colorinlistoftodos]{todonotes}

\usepackage[capitalize]{cleveref}
\crefname{enumi}{}{}
\Crefname{enumi}{Item}{Items}
\crefname{equation}{}{}
\Crefname{equation}{Equation}{Equations}

\newcommand{\itemref}[2]{\cref{#1}\cref{#1:#2}}

\numberwithin{equation}{section}

\newtheorem{theorem}[equation]{Theorem}

\newtheorem{lemma}[equation]{Lemma}
\newtheorem{proposition}[equation]{Proposition}
\newtheorem{corollary}[equation]{Corollary}

\theoremstyle{definition}
\newtheorem{remark}[equation]{Remark}
\newtheorem*{remark*}{Remark}
\newtheorem{definition}[equation]{Definition}

\newtheorem{examples}[equation]{Examples}

\newtheorem*{construction*}{Construction}

\newcommand{\Z}{\mathbb{Z}} 

\newcommand{\F}{\mathbb{F}}

\newcommand{\cP}{\mathcal{P}}
\newcommand{\cL}{\mathcal{L}}

\newcommand{\B}{\mathcal{B}}
\newcommand{\I}{\mathcal{I}}
\newcommand{\X}{\mathcal{X}}

\newcommand{\id}{\mathrm{id}}
\newcommand{\dash}{\nobreakdash-\hspace{0pt}}

\newcommand{\geom}{\mathcal{G}}
\newcommand{\geomH}{\mathcal{H}}

\newcommand{\G}{\mathbb{G}} 

\newcommand{\subhalf}{{\leavevmode \raise.5ex\hbox{\the\scriptfont0 1}\kern-.13em /\kern-.07em\lower.25ex\hbox{\the\scriptfont0 2}}}

\DeclareMathOperator{\ad}{ad}
\DeclareMathOperator{\Ann}{Ann}
\DeclareMathOperator{\Aut}{Aut}

\DeclareMathOperator{\Char}{char}

\DeclareMathOperator{\im}{im}

\DeclareMathOperator{\Miy}{Miy}

\DeclareMathOperator{\Sym}{Sym}
\DeclareMathOperator{\GL}{GL}
\DeclareMathOperator{\SL}{SL}

\newcommand{\DATwoDrawing}[6]{
		\draw (0,0)
		-- (0.5,1) node[circle,fill=black,label=-3:{#3}]  {}
		-- ++(1,2) node[circle,fill=black,label=3:{#2}]  {}
		-- ++(1,2) node[circle,fill=black,label=right:{#1}]  {}
		-- ++(1,-2) node[circle,fill=black,label=3:{#6}]  {}
		-- ++(0.5,-1) node[circle,fill=black,label=-3:{#5}]  {}
		-- ++(1,-2);
		\draw (0,3)
		-- ++(7.5,0) node[circle,fill=black,label=above:{#4}]  {}
		-- ++(-7.5,-7.5/3.5);
}

\newcommand{\AThreeDrawing}[9]{
		\useasboundingbox	(-1.0cm,1.5cm) rectangle (5.0cm,-5.3cm);
		\draw (0,0) node (1) [label=left:{#1}] {}
			++(0:2cm) node (2) [label=above:{#2}] {}
			++(0:2cm) node (3) [label=right:{#3}] {}
			++(270:2cm) node (6) [label=right:{#6}] {}
			++(180:2cm) node (5) [label=above right:{#5}] {}
			++(180:2cm) node (4) [label=left:{#4}] {}
			++(270:2cm) node (7) [label=left:{#7}] {}
			++(0:2cm) node (8) [label=below:{#8}] {}
			++(0:2cm) node (9) [label=right:{#9}] {}
		;

		\draw (1) -- (2) -- (3);
		\draw (4) -- (5) -- (6);
		\draw (7) -- (8) -- (9);
		\draw (1) -- (4) -- (7);
		\draw (2) -- (5) -- (8);
		\draw (3) -- (6) -- (9);
		\draw (1) -- (5) -- (9);
		\draw (3) -- (5) -- (7);

		\draw (1) to (6);
		\draw[shift=(1)] (6) .. controls (332:7.5cm) and (298:7.5cm)	.. (8);
		\draw (8) to (1);

		\draw (3) to (4);
		\draw[shift=(3)] (4) .. controls (208:7.5cm) and (242:7.5cm)	.. (8);
		\draw (8) to (3);

		\draw (7) to (2);
		\draw[shift=(7)] (2) .. controls (62:7.5cm) and (28:7.5cm)	.. (6);
		\draw (6) to (7);

		\draw (9) to (2);
		\draw[shift=(9)] (2) .. controls (118:7.5cm) and (152:7.5cm)	.. (4);
		\draw (4) to (9);
}

\begin{document}

\title{Matsuo algebras in characteristic 2}

\author{Tom De Medts}
\address[Tom De Medts]{Ghent University \\
Department of Mathematics: Algebra and Geometry \\
\mbox{Krijgslaan} 281--S25 \\
9000 Gent \\
Belgium}
\email{tom.demedts@ugent.be}

\author{Mathias Stout}
\address[Mathias Stout]{KU Leuven \\
Department of Mathematics \\
Celestijnenlaan 200B bus 2400 \\
B-3001 Leuven \\
Belgium}
\email{mathias.stout@kuleuven.be}

\date{August 22, 2023}

\begin{abstract}
    We extend the theory of Matsuo algebras, which are certain non-associative algebras related to 3-transposition groups, to characteristic 2. Instead of idempotent elements associated to points in the corresponding Fischer space, our algebras are now generated by nilpotent elements associated to lines. For many 3-transposition groups, this still gives rise to a $\Z/2\Z$-graded fusion law, and we provide a complete classification of when this occurs.
    In one particular small case, arising from the 3-transposition group $\Sym(4)$, the fusion law is even stronger, and the resulting Miyamoto group is an algebraic group $\G_a^2 \rtimes \G_m$.
\end{abstract}

\keywords{Non-associative algebras, Matsuo algebras, axial algebras, decomposition algebras, $3$-transposition groups, Fischer spaces}
\makeatletter 
\@namedef{subjclassname@2020}{\textup{2020} Mathematics Subject Classification}
\makeatother
\subjclass[2020]{20B25, 20F29, 17A99, 51E30}

\maketitle

\section{Introduction}

\emph{Axial algebras} have been introduced by Jon Hall, Felix Rehren and Sergey Shpectorov in \cite{HRS15}. They form a class of non-associative algebras generated by \emph{idempotents}, such that each idempotent gives rise to a decomposition of the algebra respecting a certain \emph{fusion law}.

One of the early examples that has been studied in more detail, are the so-called \emph{Matsuo algebras}, related to $3$-transposition groups. More precisely, if $(G, D)$ is a $3$-transposition group (where $G$ is the group and $D$ is the distinguished generating set of involutions), then for each commutative field $k$ with $\Char(k)\neq 2$ and each $\eta \in k \setminus \{ 0,1 \}$, there is a $k$-algebra $A$ \emph{spanned} by idempotents corresponding to the elements of $D$, with a very precise multiplication table (see \cref{def:matsuo-not-2} below). Each of these generating idempotents $e$ then gives rise to a decomposition of $A$ into eigenspaces for the left multiplication by $e$, and the multiplication of these eigenspaces satisfies the \emph{Jordan fusion law}, making it very similar in structure as the Peirce decomposition for Jordan algebras.
(We refer to \cite{DMR17} for a deeper connection between Matsuo algebras and Jordan algebras.)

Extending the theory to characteristic $2$ is not obvious. First, the decomposition into eigenspaces for an idempotent can never have a $\Z/2\Z$-graded fusion law, as we observe in \cref{pr:axial2} below. Therefore, it seems more natural to use a spanning set of \emph{nilpotent} elements rather than idempotents. (In this case, the role of the parameter $\eta$ becomes irrelevant.) Even so, just using nilpotents associated to the elements of $D$ gives rise to highly uninteresting decompositions, because the multiplication by such a nilpotent turns out to be a nilpotent operator (and hence has $0$ as its only eigenvalue).

We solve this by using nilpotent elements associated to \emph{lines} in the Fischer space associated to the $3$-transposition group, or put differently, to triples of involutions contained in a common $\Sym(3)$-subgroup of $G$. Somewhat surprisingly, this does indeed give rise to interesting decompositions in many cases.

Perhaps the most interesting case is the smallest non-trivial case arising from the $3$\dash transposition group $\Sym(4)$ with as generating set $D$ the set of all $6$ transpositions. In this case, we have only $4$ lines, but the decompositions satisfy a fusion law stronger than a $\Z/2\Z$-grading and allows for the action of a multiplicative torus $\G_m$ on each decomposition. Together, these four tori generate a group of the form $\G_a^2 \rtimes \G_m$; see \cref{th:CQ-Miy}. We also compute the full automorphism group of this algebra in \cref{pr:DAG22-Aut}.

For Fischer spaces of \emph{symplectic type}, we get (ordinary) $\Z/2\Z$-graded decompositions, as we show in \cref{th:symplectic}.
%
If the Fischer space is not of symplectic type, then additional problems arise. In the smallest case, which is the affine plane of order $3$, associated to the $3$\dash transposition group $3^2 : 2$, we have to pass to the \emph{generalized eigenspaces} (rather than the eigenspaces), which does not fit with the original definition of axial algebras from \cite{HRS15} but still complies with the more general framework of \emph{decomposition algebras} developed in \cite{DMPSVC20}. In this specific case, the resulting decomposition algebra is again $\Z/2\Z$-graded; see \cref{th:aff-decalg}.

However, as soon as the Fischer space is not of symplectic type and has rank~$\geq 4$, then things break down and the resulting nilpotent Matsuo algebra is no longer $\Z/2\Z$-graded, as we show in \cref{th:main} below.

\medskip

When we had almost finished writing this paper, we became aware of the surprising fact that in \cite{CHPS12}, the authors study exactly the same algebras as we do, but only over the field $\F_2$. They had different goals than we do, so our results are almost completely disjoint from each other.

\section{Axial decomposition algebras}

We first recall the definition of (commutative) axial decomposition algebras from \cite{DMPSVC20}.
All algebras will be \emph{commutative}, but \emph{not necessarily associative nor unital} algebras over a commutative base ring $R$.

\begin{definition}\label{def:fus}
    \begin{enumerate}
        \item 
            A \emph{fusion law} is a pair $(X, *)$ where $X$ is a (usually finite) set and $*$ is a map from $X \times X$ to $2^X$, where $2^X$ denotes the power set of $X$.
            A fusion law $(X, *)$ is called \emph{symmetric} if $x * y = y * x$ for all $x,y \in X$.
            We call $e \in X$ a \emph{unit} if $e * x \subseteq \{ x \}$ and $x * e \subseteq \{ x \}$ for all $x \in X$.
        \item\label{def:fus:jordan}
        	The \emph{Jordan fusion law} is the symmetric fusion law $(X, *)$ with $X = \{ e, z, h \}$ and with $*$ given by the following table:
        	\begin{center}
        		\renewcommand{\arraystretch}{1.2}
        		\setlength{\tabcolsep}{0.75em}
        		\begin{tabular}[b]{c|ccc}
        			$*$	& $e$ & $z$ & $h$ \\
        			\hline
        			$e$ & $\{ e \}$ & $\emptyset$ & $\{ h \}$ \\
        			$z$ & $\emptyset$ & $\{ z \}$ & $\{ h \}$ \\
        			$h$ & $\{ h \}$ & $\{ h \}$ & $\{ e,z \}$
        		\end{tabular}
        	\end{center}
    \end{enumerate}
\end{definition}

\begin{definition}
    Let $\Phi = (X, *)$ be a fusion law.
    \begin{enumerate}[(i)]
        \item
          A \emph{$\Phi$-decomposition} of an $R$-algebra $A$ is a direct sum decomposition $A = \bigoplus_{x \in X} A_x$ (as $R$-modules) such that $A_x A_y \subseteq A_{x * y}$ for all $x,y \in X$, where $A_{Y} := \bigoplus_{y \in Y} A_y$ for all $Y \subseteq X$.
        \item
            A \emph{$\Phi$-decomposition algebra} is a triple $(A, \I, \Omega)$ where $A$ is an $R$-algebra, $\I$ is an index set and $\Omega$ is a tuple of $\Phi$-decompositions of $A$ indexed by $\I$.
            We will usually write the corresponding decompositions as $A = \bigoplus_{x \in X} A_x^i$, so
            \[ \Omega = \bigl( ( A_x^i )_{x \in X} \mid i \in \I \bigr) . \]
    \end{enumerate}
\end{definition}

\begin{definition}
    Let $\Phi = (X, *)$ be a fusion law with a distinguished unit $e \in X$.
    For each $x \in X$, let $\lambda_x \in R$.
    A $\Phi$-decomposition algebra $(A, \I, \Omega)$ will be called \emph{axial} (with \emph{parameters} $\lambda_x$)
    if for each $i \in \I$, there is some non-zero $a_i \in A_e^i$ (called an \emph{axis}) such that:
    \begin{equation}\label{eq:axial}
        a_i \cdot b = \lambda_x b \quad \text{for all } x \in X \text{ and for all } b \in A_x^i .
    \end{equation}
    In particular, $A$ decomposes into eigenspaces for the left multiplication operator
    \[ \ad_{a_i} \colon A \to A \colon b \mapsto a_i \cdot b . \]
    (Notice, however, that some of the parameters $\lambda_x$ might coincide, so the summands $A_x^i$ are not necessarily eigenspaces.)
\end{definition}

The most interesting examples of decomposition algebras arise from fusion laws with a \emph{grading}.

\begin{definition}\label{def:grading}
    \begin{enumerate}
        \item 
        	Let $\Gamma$ be a group.
        	Then the map
        	\[ * \colon \Gamma \times \Gamma \to 2^\Gamma \colon (g, h) \mapsto \{ g h \} \]
        	makes $(\Gamma, *)$ into a fusion law, which we call a \emph{group fusion law}.
        	The identity element of $\Gamma$ is the unique unit of the fusion law $(\Gamma, *)$.
        \item
            Let $(X, *)$ and $(Y, *)$ be two fusion laws.
            A \emph{morphism} from $(X, *)$ to $(Y, *)$ is a map $\xi \colon X \to Y$ such that
            \[ \xi(x_1 * x_2) \subseteq \xi(x_1) * \xi(x_2) \]
            for all $x_1, x_2 \in X$,
            where we have denoted the obvious extension of $\xi$ to a map $2^X \to 2^Y$ also by $\xi$.
        \item
            Let $(X, *)$ be a fusion law and let $(\Gamma, *)$ be a group fusion law.
            A \emph{$\Gamma$-grading} of $(X, *)$ is a morphism $\xi \colon (X, *) \to (\Gamma, *)$.
            In this paper, the grading group $\Gamma$ will always be abelian.
    \end{enumerate}
\end{definition}

The Jordan fusion law from \itemref{def:fus}{jordan} is $\Z/2\Z$-graded. Indeed, 
the map $\xi \colon X \to \Z/2\Z$ mapping $e$ and $z$ to $0$ and $h$ to $1$ 
is a fusion law morphism.
If the fusion law is graded by a non-trivial group, then we can construct an interesting subgroup of the automorphism group $\Aut(A)$ of the algebra $A$.

\begin{definition}\label{def:miy}
    \begin{enumerate}
        \item 
            Let $R^\times$ be the group of invertible elements of the base ring~$R$.
            An \textit{$R$-character} of $\Gamma$ is a group homomorphism $\chi \colon \Gamma \to R^\times$.
            The \textit{$R$-character group} of $\Gamma$ is the group $\X$ consisting of all $R$-characters of $\Gamma$,
            with group operation induced by multiplication in $R^\times$.
        \item
            Let $(A, \I, \Omega)$ be a $\Gamma$-decomposition algebra.
            Let $\chi \in \X$.
            For each decomposition $(A^i_g)_{g \in \Gamma} \in \Omega$, we define a linear map
            \[ \tau_{i,\chi} \colon A \to A \colon a \mapsto \chi(g) a \quad \text{for all } a \in A^i_g ; \]
            we call this a \emph{Miyamoto map}.
            It follows immediately from the definitions that each $\tau_{i,\chi}$ is an automorphism of the $R$-algebra $A$.
        \item
            We define the \emph{(full) Miyamoto group} of $(A, \I, \Omega)$ by
            \[ \Miy(A, \I, \Omega) := \langle \tau_{i,\chi} \mid i \in \I, \chi \in \X \rangle \leq \Aut(A) . \]        
    \end{enumerate}
\end{definition}

\section{Matsuo algebras}    

One of the first classes of examples of axial decomposition algebras with a Jordan fusion law that has been systematically studied, is the class of \emph{Matsuo algebras} arising from $3$-transposition groups.
We recall their definition.

\begin{definition}\label{def:matsuo-not-2}
    \begin{enumerate}
        \item
            A \emph{3-transposition group} is a pair $(G,D)$ where $G$ is a group and $D \subset G$ is a conjugacy class of involutions in $G$ such that $G = \langle D \rangle$ and such that for all $d,e \in D$, the order $o(de)$ of the element $de \in G$ is at most $3$.
        \item
            Assume that $2$ is invertible in $R$, and let $\eta \in R$ be arbitrary. Let $A$ be a free $R$-module with basis $D$.
            We make $A$ into an $R$-algebra by declaring
            \[ d \cdot e := \begin{cases}
                d & \text{if } d = e, \\
                0 & \text{if } o(de) = 2, \\
                \tfrac{\eta}{2} (d + e - f) & \text{if } o(de) = 3, \text{ where $f = ded = ede$}.
            \end{cases}
            \]
            This algebra will be denoted by $M_R(G, D, \eta)$ and will be called the \emph{Matsuo algebra} (over $R$) of the $3$-transposition group $(G,D)$ with parameter $\eta$.
    \end{enumerate}
\end{definition}

The 3-transposition groups have a geometric counterpart known as \emph{Fischer spaces}.

\begin{definition}
    \begin{enumerate}
        \item 
        	A {\em partial linear space} $\geom$ is a pair $\geom = (\cP, \cL)$, where $\cP$ is a set of {\em points} and $\cL \subseteq 2^\cP$ is a set of {\em lines} such that any $\ell \in \cL$ has size at least~$2$ and such that any two distinct lines intersect in at most one point.
        	Two distinct points $x,y \in \cP$ are called \emph{collinear} if they are contained in a (necessarily unique) common line, and we denote this by $x \sim y$.
        \item
        	A partial linear space is called a {\em partial triple system} if every line has exactly $3$ points.
            In this case, for any two collinear points $x,y \in \cP$, the unique third point on the line through $x$ and $y$ will be denoted by $x \wedge y$.
        \item
            A non-empty subset $S$ of $\cP$ is called a {\em subspace} of $\geom$ if it is closed under the operation $\wedge$, and if $T$ is any subset of $\cP$, then we write $\langle T \rangle$ for the {\em subspace generated by $T$}, i.e., the smallest subspace of $\geom$ containing $T$.
        \item
            Assume now that $(G,D)$ is a $3$-transposition group. Then we can associate a partial triple system to $(G,D)$ with point set $D$, where a triple $\{ d, e, f \} \subseteq D$ forms a line if and only if $d,e,f$ are the three involutions in a $\Sym(3)$-subgroup of $G$. In particular, two distinct points $d$ and $e$ are collinear if and only if $o(de) = 3$ in $G$.
            Such a partial triple system is called a \emph{Fischer space}.
    \end{enumerate}
\end{definition}
\begin{remark}\label{rem:connected}
    Sometimes, the defining set $D$ of a $3$-transposition group $(G,D)$ is allowed to be a union of several conjugacy classes, but our assumption is not a severe one and is common. It has as a consequence that the resulting Fischer space is \emph{connected}.
    See, e.g., \cite[Theorem (3.1)]{CH95}.
\end{remark}
\begin{definition}
    The \emph{rank} of a Fischer space $\geom$ is the maximal length $r$ of a chain of connected subspaces
    \[ \geom_0 < \geom_1 < \dots < \geom_r = \geom . \]
    In particular, a Fischer space of rank $1$ is a point and a Fischer space of rank $2$ is a line (with three points).
\end{definition}

The following two examples play an important role, as they are the only two Fischer spaces of rank $3$; see \cref{th:rank3}.
\begin{examples}\label{ex:fischer}
    \begin{enumerate}
        \item\label{ex:fischer:DA2}
            The $3$-transposition group $G = \Sym(4)$ with generating set of involutions $D$ consisting of the six transpositions in $\Sym(4)$ has as associated Fischer space the dual affine plane of order $2$, also known as the \emph{complete quadrilateral}:
            
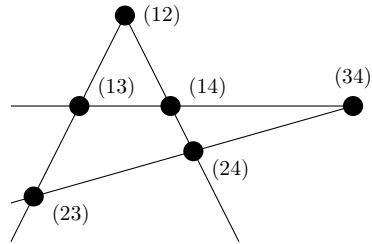
\begin{center}
            \begin{tikzpicture}[scale=.6, every node/.style={scale=.8}]
                \DATwoDrawing{$(12)$}{$(13)$}{$(23)$}{$(34)$}{$(24)$}{$(14)$}
            \end{tikzpicture}
            \end{center}
        \item\label{ex:fischer:A3}
            The $3$-transposition group $G = 3^2 \colon 2$ (where the action of $C_2$ on $C_3 \times C_3$ is given by inversion) with generating set $D$ consisting of all $9$ involutions in $G$, has as associated Fischer space the affine plane of order $3$. We will refer to this example of a Fischer space simply as \emph{the affine plane}.
            
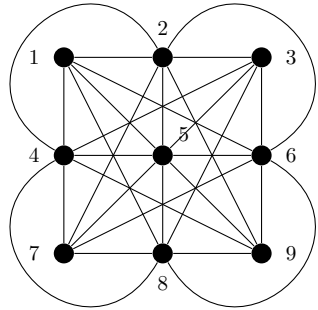
\begin{center}
        	\begin{tikzpicture}[scale=.65,every node/.style={scale=.8,fill=black,circle}]
        		\AThreeDrawing{$1$}{$2$}{$3$}{$4$}{$5$}{$6$}{$7$}{$8$}{$9$}
        	\end{tikzpicture}
            \end{center}
    \end{enumerate}
\end{examples}

\begin{theorem}\label{th:rank3}
    Let $\geom$ be a connected Fischer space of rank $3$.
    Then $\geom$ is isomorphic to one of the two examples from \cref{ex:fischer}.
    
    In particular, if $\geom$ is a Fischer space of arbitrary rank, then the subspace generated by two distinct intersecting lines of $\geom$ is isomorphic to one of these two examples.
\end{theorem}
\begin{proof}
    See, e.g., \cite[Section 3.2]{CH95}.
    In fact, this property is often used as an alternative equivalent \emph{definition} of Fischer spaces, as in \cite[p.\@~343]{CH92}.
\end{proof}
\begin{definition}\label{def:symplectic}
    A Fischer space $\geom$ is said to be of \emph{symplectic type} if all subspaces generated by two distinct intersecting lines are complete quadrilaterals.
\end{definition}

It is often useful to go directly from the Fischer space to the Matsuo algebra.
\begin{definition}
    Let $\geom = ( \cP, \cL )$ be a Fischer space. Then we define the Matsuo algebra $M_R(\geom, \eta)$ as the $R$-algebra with underlying free $R$-module with basis $\cP$, with multiplication given by
        \[ x \cdot y := \begin{cases}
            x & \text{if } x = y, \\
            0 & \text{if } x \neq y \text{ and } x \not\sim y, \\
            \tfrac{\eta}{2} (x + y - x \wedge y) & \text{if } x \sim y.
        \end{cases}
        \]
\end{definition}

For the rest of the paper, this is the point of view we will adopt, so all our algebras will be defined starting from a Fischer space rather than from the underlying $3$\dash transposition group.

Matsuo algebras are examples of axial decomposition algebras, with all elements of the basis $\cP$ as axes.
More precisely, for each $x \in \cP$, there is a corresponding decomposition of $A = M_R(\geom, \eta)$ into three parts, with parameters $0$, $1$ and $\eta$. When $\eta \not\in \{ 0,1 \}$, these summands are precisely equal to the eigenspaces of the multiplication operator $\ad_x$.

\section{Axial algebras in characteristic 2}

In the original definition of axial algebras as in \cite{HRS15}, the axes were always assumed to be \emph{idempotent} elements, and this is also how the Matsuo algebras had originally been studied.
In terms of the defining parameters in \cref{eq:axial}, this corresponds to the assumption that $\lambda_e = 1$.
Often, however, this is an unnecessary and sometimes unwanted restriction; in fact, if $2 = 0$ in $R$, this essentially excludes interesting decomposition algebras with a $\Z/2\Z$-grading, as the following observation shows.
Recall that all algebras are assumed to be commutative.
\begin{proposition}\label{pr:axial2}
    Assume that $2 = 0$ in $R$ and let $A$ be an axial decomposition algebra over $R$ with respect to a $\Z/2\Z$-graded fusion law $(X, *)$ such that all axes are idempotents.
    Then all axes are contained in the \emph{zero-core}, i.e., the intersection of the $0$-graded part of all of the decompositions.
    
    In particular, the algebra $A$ is not generated by its axes unless all its $\Z/2\Z$\dash decompositions are trivial.
\end{proposition}
\begin{proof}
    Let $a,b \in A$ be two axes, and decompose $b$ with respect to the $\Z/2\Z$\dash decomposition corresponding to $a$, as $b = b_0 + b_1$, with $b_0 \in A_0^a$ and $b_1 \in A_1^a$.
    Then, since $2 = 0$, we have
    \[ b = b^2 = (b_0 + b_1)^2 = b_0^2 + b_1^2 \in A_0^a \]
    because both $A_0^a \cdot A_0^a \subseteq A_0^a$ and $A_1^a \cdot A_1^a \subseteq A_0^a$.
\end{proof}

Because of \cref{pr:axial2}, we will, instead, assume that our axes are \emph{nilpotent} elements. We point out that this also makes sense for the deeper reason that the origin of the theory of axial algebras (and Matsuo algebras in particular) is found in the theory of VOAs, where an important role is played by elements $e$ for which $e^2 = 2e$ (rather than $e^2 = e$). This seems like an innocent rescaling but further illustrates that over arbitrary rings, the assumption $e^2 = 2e$ is more natural than $e^2 = e$ and naturally leads to nilpotent elements in the case that $2 = 0$.

\section{Nilpotent Matsuo algebras}

Assume, from now on, that $2 = 0$ in $R$.
In view of the previous section, the following definition is natural.
\begin{definition}
    Let $\geom = ( \cP, \cL )$ be a Fischer space. Assume that $2 = 0$ in~$R$ and let $\eta \in R$ with $\eta \neq 0$.
    Then we define the \emph{nilpotent Matsuo algebra} $M_R(\geom, \eta)$ as the $R$-algebra with underlying free $R$-module with basis $\cP$, with multiplication given by
        \[ x \cdot y := \begin{cases}
            0 & \text{if } x = y \text{ or } x \not\sim y, \\
            \eta (x + y + x \wedge y) & \text{if } x \sim y.
        \end{cases}
        \]
\end{definition}
\begin{remark}\label{rem:indep}
    \begin{enumerate}
        \item\label{rem:indep:1}
            For each line $\ell \in \cL$, the product of two distinct points on $\ell$ is independent of the choice of these points; it is always equal to $\eta$ times the sum of these three points.
        \item
            The role of the parameter $\eta$ is negligible. In fact, if $R$ is a field, then we can simply replace each point $x$ by $\eta^{-1} x$ and get the same multiplication on $A$ with $\eta$ replaced by $1$.
    \end{enumerate}
\end{remark}


From now on, we keep the assumptions that $\geom = ( \cP, \cL )$ is a Fischer space, that $2 = 0$ in~$R$ and that $A = M_R(\geom, 1)$ is the corresponding nilpotent Matsuo algebra (so $\eta = 1$).

A natural attempt now is to construct decompositions corresponding to the elements of the basis $\cP$. The algebra $A$ no longer decomposes into eigenspaces for $\ad_x$, but even passing to generalized eigenspaces does not work, because of the following observation.

\begin{proposition}
    For each $x \in \cP$, the operator $\ad_x$ is nilpotent. More precisely, we have $\ad_x^2 = 0$.
\end{proposition}
\begin{proof}
    It suffices to show that $x \cdot xy = 0$ for each $y \in \cP$.
    If $x = y$ or $x \not\sim y$, this is obvious, so assume that $x \sim y$.
    Then $x \cdot xy = x(x + y + x \wedge y) = xy + x(x \wedge y)$.
    Since $xy = x(x \wedge y)$ by \itemref{rem:indep}{1}, the result follows.
\end{proof}

However, not only the points $x \in \cP$ are nilpotent, but in fact, every single element in $A$ has square $0$. It turns out that the \emph{line nilpotents} are a better candidate to obtain decompositions of the algebra $A$.

\begin{definition} \label{def:line-nilpotent}
    For each line $\ell \in \cL$, we let $s_\ell \in A$ be the sum of the three points on~$\ell$, and we call this a \emph{line nilpotent} of $A$.
    In fact, we will simply write $\ell$ for $s_\ell$; this will not cause confusion.
\end{definition}

We first collect some easy computations that we will use repeatedly.
Recall the two examples from \cref{ex:fischer}.

\begin{lemma}\label{le:lx}
    Let $\ell \in \cL$ and $x \in \cP$.
    \begin{enumerate}
        \item\label{le:lx:1} 
            If $x \in \ell$ or if $x$ is not collinear to any point of $\ell$, then $x\ell = 0$ in $A$.
        \item\label{le:lx:2}
            If the subspace spanned by $x$ and $\ell$ is a complete quadrilateral, then there are precisely two lines $m,n$ joining $x$ with $\ell$; in this case, $x\ell = m + n$ in $A$. We also have $x\ell = \ell + \ell'$ where $\ell'$ is the unique line in this subspace different from $\ell$ and not containing $x$.
        \item\label{le:lx:3}
            If the subspace spanned by $x$ and $\ell$ is an affine plane, then there is a unique line $m \neq \ell$ in this subspace parallel to $\ell$ and not containing $x$; in this case, $x\ell = x + \ell + m$ in~$A$.
    \end{enumerate}
\end{lemma}
\begin{proof}
    This is straightforward.
\end{proof}

\begin{lemma}\label{le:lm}
    Let $\ell,m \in \cL$.
    \begin{enumerate}
        \item\label{le:lm:1}
            If $\ell = m$, then $\ell m = 0$ in $A$.
        \item\label{le:lm:2}
            If $\ell \cap m = \{ x \}$ and they span a complete quadrilateral, then $\ell m = \ell + m$ in~$A$.
        \item\label{le:lm:3}
            If $\ell \cap m = \{ x \}$ and they span an affine plane, then $\ell m = a+b+c+d$ in~$A$, where $a,b,c,d$ are the four points in this subspace not on $\ell$ nor on $m$.
        \item\label{le:lm:4}
            If $\ell \cap m = \emptyset$ and they span an affine plane, then $\ell m = s$ in~$A$, where $s$ is the sum of all $9$ points in this subspace.
    \end{enumerate}
\end{lemma}
\begin{proof}
    This is straightforward.
\end{proof}

Nilpotent Matsuo algebras are always somewhat degenerate, as they contain annihilating elements.
\begin{definition}
    An element $a \in A$ is \emph{annihilating} if $ab = 0$ for all $b \in A$. We write
    \[ \Ann(A) := \{ a \in A \mid a \text{ is annihilating} \} . \]
\end{definition}
\begin{lemma}\label{le:ann}
    Let $s$ be the sum of all points in $\cP$.
    Then $\Ann(A) = \langle s \rangle$.
\end{lemma}
\begin{proof}
    To show that $s \in \Ann(A)$, it suffices to show that $sx = 0$ for all $x \in \cP$. Let $\cL_x$ be the set of lines containing $x$. As $yx = 0$ whenever $y=x$ or $y \not\sim x$, the product $sx$ is equal to the sum $\sum_{\ell \in \cL_x} \ell x$, which is $0$ by \itemref{le:lx}{1}.
    
    Conversely, assume that $a \in \Ann(A)$ and write $a = \sum_{p \in \cP} a_p p$ with $a_p \in R$. Let $p \sim q$ be any pair of distinct collinear points in $\geom$ and let $x = p \wedge q$. Then the fact that $ax = 0$ implies that $a_p = a_q$. Because $\geom$ is connected by \cref{rem:connected}, we conclude that all $a_p$ are equal and hence $a$ is a multiple of $s$.
\end{proof}
\begin{definition}
    We define the \emph{reduced nilpotent Matsuo algebra} to be the quotient $A' := A / \Ann(A)$.
\end{definition}

We will now first analyze the two Fischer spaces from \cref{ex:fischer}.

\subsection{The complete quadrilateral}\label{ss:CQ}

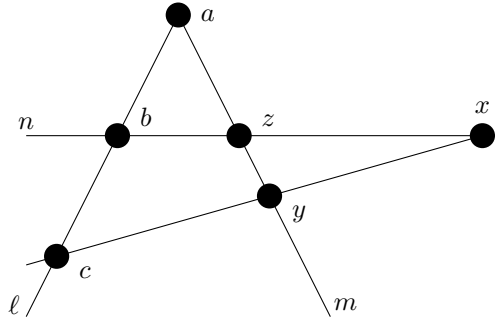
\begin{figure}[ht!]
    \begin{tikzpicture}[scale=.8]
        \DATwoDrawing{$a$}{$b$}{$c$}{$x$}{$y$}{$z$}
        \node at (-.2,.2) {$\ell$};
        \node at (5.25,.2) {$m$};
        \node at (0,3.2) {$n$};
    \end{tikzpicture}
    \caption{The complete quadrilateral.}\label{fig:CQ}
\end{figure}

Let $\geom$ be the complete quadrilateral labeled as in \cref{fig:CQ} and let $\ell$ be the line $\{ a,b,c \}$. We will analyze the decomposition of $A$ arising from the operator $\ad_\ell$.
For our computations, we will also need the lines $m = \{ a,y,z \}$ and $n = \{ b,x,z \}$ as in \cref{fig:CQ}.

We begin with a lemma that is useful also for Fischer spaces of symplectic type of higher rank; see \cref{def:symplectic}.
\begin{lemma}\label{le:lla=la}
    Let $\geom$ be a Fischer space of symplectic type and let $\ell \in \cL$.
    Then the operator $\ad_\ell$ is idempotent, i.e., $\ell \cdot \ell a = \ell a$ for all $a \in A$.
\end{lemma}
\begin{proof}
    It suffices to show that $\ell \cdot \ell x = \ell x$ for all $x \in \cP$, and this is obvious if $x \in \ell$ or if $x$ is not collinear with any point of $\ell$.
    Assume now that $x$ is collinear with a point of $\ell$ and consider the subspace generated by $\ell$ and $x$.
    By assumption, this subspace is a complete quadrilateral, so we may assume that $\ell$ and $x$ are as in \cref{fig:CQ}.
    Then $\ell x = \ell + m$ and hence, by \itemref{le:lx}{2} and \itemref{le:lm}{2},
    \[ \ell \cdot \ell x = \ell(\ell + m) = \ell m = \ell + m = \ell x. \qedhere \]
\end{proof}
\begin{remark}
    The statement of \cref{le:lla=la} is false for the affine plane, so it is false for any Fischer space that is \emph{not} of symplectic type.
\end{remark}

We now assume again that $\geom$ is a complete quadrilateral.
As before, let $s = a+b+c+x+y+z \in A$ be the sum of all points.
\begin{proposition}\label{pr:des-decs}
    We have a decomposition $A = A^\ell_0 \oplus A^\ell_1$ into eigenspaces for the operator $\ad_\ell$ with eigenvalues $0$ and $1$, respectively, where
    \begin{align*}
        A^\ell_0 &= \langle a, b, c, x+y+z \rangle = \langle a, b, c, s \rangle , \\
        A^\ell_1 &= \langle \ell x, \ell y \rangle .
    \end{align*}
\end{proposition}
\begin{proof}
    We have $\ell a = \ell b = \ell c = 0$ by \itemref{le:lx}{1} and $\ell s = 0$ by \cref{le:ann}.
    Hence
    \[ \langle a, b, c, s \rangle \leq A^\ell_0 . \]
    By \cref{le:lla=la}, we already know that
    \[ \langle \ell x, \ell y \rangle \leq A^\ell_1 . \]
    Notice that $\ell x$ and $\ell y$ are linearly independent because $\ell x = b + c + y + z$ and $\ell y = a + c + x + z$.
    By dimension count, the result follows.
\end{proof}
\begin{remark}
    The subspace $A_1^\ell$ turns out to be independent of the choice of $\ell$. Indeed, it is equal to $\langle s + a + x, s + b + y, s + c + z \rangle$ (notice that the sum of these three vectors is $0$) and this is independent of the choice of $\ell$.
\end{remark}

\begin{theorem}\label{th:des-decalg}
    Let $\geom$ be the complete quadrilateral. Then the nilpotent Matsuo algebra $A = M_R(\geom, 1)$ is a $6$-dimensional decomposition algebra with fusion law
	\begin{center}
		\renewcommand{\arraystretch}{1.2}
		\setlength{\tabcolsep}{0.75em}
		\begin{tabular}[b]{c|ccc}
			$*$	& $0$ & $1$ \\
			\hline
			$0$ & $\{ 0 \}$ & $\{ 1 \}$ \\
			$1$ & $\{ 1 \}$ & $\emptyset$
		\end{tabular},
	\end{center}
    where for each of the $4$ lines $\ell \in \cL$, there is a decomposition $A = A^\ell_0 \oplus A^\ell_1$ into eigenspaces for $\ad_\ell$ with eigenvalues $0$ and $1$, respectively.
    The subspaces $A^\ell_1$ coincide for all $4$ choices of $\ell \in \cL$.
\end{theorem}
\begin{proof}
    We keep the notations from \cref{pr:des-decs}. It remains to show that $A^\ell_0$ is a subalgebra of $A$, that $A^\ell_0 \cdot A^\ell_1 \subseteq A^\ell_1$ and that $A^\ell_1 \cdot A^\ell_1 = 0$.
    
    To show that $A^\ell_0$ is a subalgebra of $A$, we first observe that $\langle a, b, c \rangle$ is a subalgebra of $A$ because the line $\ell = \{ a,b,c \}$ is a subspace of $\geom$. Since $s \in \Ann(A)$ by \cref{le:ann}, this implies immediately that $\langle a,b,c,s \rangle$ is indeed a subalgebra of $A$.
    
    Next, we compute that
    \begin{align*}
        a \cdot \ell x &= a (\ell + m) = a \ell + am = 0 \quad \text{and} \\
        a \cdot \ell y &= a (\ell + n) = an = \ell + m = \ell x \in A^\ell_1.
    \end{align*}
    
    Finally, we have
    \[ \ell x \cdot \ell y = (\ell + m)(\ell + n) = \ell m + \ell n + mn = (\ell + m) + (\ell + n) + (m + n) = 0 . \qedhere \]
\end{proof}

\begin{corollary}\label{cor:des-decalg'}
    Let $\geom$ be the complete quadrilateral and let $A = M_R(\geom, 1)$. Then the reduced nilpotent Matsuo algebra $A' = A / \Ann(A)$ is a $5$-dimensional decomposition algebra with fusion law
	\begin{center}
		\renewcommand{\arraystretch}{1.2}
		\setlength{\tabcolsep}{0.75em}
		\begin{tabular}[b]{c|ccc}
			$*$	& $0$ & $1$ \\
			\hline
			$0$ & $\{ 0 \}$ & $\{ 1 \}$ \\
			$1$ & $\{ 1 \}$ & $\emptyset$
		\end{tabular},
	\end{center}
    where for each of the $4$ lines $\ell \in \cL$, there is a decomposition $A' = A'^\ell_0 \oplus A'^\ell_1$ into eigenspaces for $\ad_\ell$ with eigenvalues $0$ and $1$, respectively.
    The subspaces $A'^\ell_1$ coincide for all $4$ choices of $\ell \in \cL$.
\end{corollary}

\subsection{The affine plane of order 3}

The situation for the affine plane is more complicated.
Again, the only eigenvalues for $\ad_\ell$ will be $0$ and $1$, but the operator $\ad_\ell$ is not semisimple, so we need to pass to the generalized eigenspaces instead.

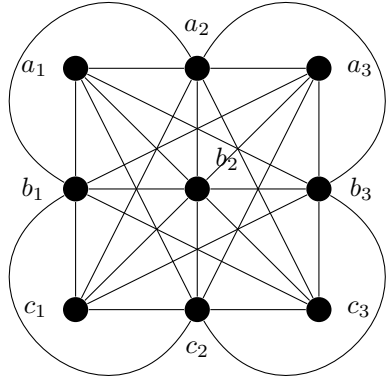
\begin{center}
\begin{tikzpicture}[scale=0.8,every node/.style={fill=black,circle}]
	\AThreeDrawing{$a_1$}{$a_2$}{$a_3$}{$b_1$}{$b_2$}{$b_3$}{$c_1$}{$c_2$}{$c_3$}
\end{tikzpicture}
\end{center}

Let $\geom$ be the affine plane labeled as above and let $\ell$ be the line $\{ a_1, a_2, a_3 \}$. We also set $m = \{ b_1, b_2, b_3 \}$ and $n = \{ c_1, c_2, c_3 \}$.
Let $s$ be the sum of all $9$ points of $\geom$.

\begin{proposition}\label{pr:aff-decs}
    The operator $\ad_\ell$ has eigenvalues $0$ and $1$, with corresponding eigenspaces
    \begin{align*}
        A^\ell_0 &= \langle a_1, a_2, a_3, s \rangle , \\
        A^\ell_1 &= \langle b_1 + b_2, b_1 + b_3, c_1 + c_2, c_1 + c_3 \rangle
    \end{align*}
    and generalized eigenspaces
    \begin{align*}
        \widetilde A^\ell_0 &= \langle a_1, a_2, a_3, m, n \rangle , \\
        \widetilde A^\ell_1 &= A^\ell_1 .
    \end{align*}
    In particular, $A = \widetilde A^\ell_0 \oplus A^\ell_1$.
\end{proposition}
\begin{proof}
    By \itemref{le:lx}{1} and \cref{le:ann}, we already know that
    \[ \langle a_1, a_2, a_3, s \rangle \leq A^\ell_0 . \]
    Next, we have $b_i \ell = b_i + \ell + n$ and $c_i \ell = c_i + \ell + m$ for each $i$ by \itemref{le:lx}{3}, so indeed $(b_i + b_j) \ell = b_i + b_j$ and $(c_i + c_j) \ell = c_i + c_j$ for all $i \neq j$, hence
    \[ \langle b_1 + b_2, b_1 + b_3, c_1 + c_2, c_1 + c_3 \rangle \leq A^\ell_1 . \]
    Finally, by \itemref{le:lm}{4}, we have
    \[ \ell m = \ell n = s \in A^\ell_0 , \]
    hence $m,n \in \widetilde A^\ell_0$. (Notice that $a_1 + a_2 + a_3 + m + n = s$.)
    By dimension count, the result follows.
\end{proof}

\begin{theorem}\label{th:aff-decalg}
    Let $\geom$ be the affine plane of order $3$. Then the nilpotent Matsuo algebra $A = M_R(\geom, 1)$ is a $9$-dimensional decomposition algebra with fusion law
	\begin{center}
		\renewcommand{\arraystretch}{1.2}
		\setlength{\tabcolsep}{0.75em}
		\begin{tabular}[b]{c|ccc}
			$*$	& $0$ & $1$ \\
			\hline
			$0$ & $\{ 0 \}$ & $\{ 1 \}$ \\
			$1$ & $\{ 1 \}$ & $\{ 0 \}$
		\end{tabular},
	\end{center}
    where for each of the $12$ lines $\ell \in \cL$, there is a decomposition $A = \widetilde A^\ell_0 \oplus A^\ell_1$ into generalized eigenspaces for $\ad_\ell$ with eigenvalues $0$ and $1$, respectively.
\end{theorem}
\begin{proof}
    We keep the notations from \cref{pr:aff-decs}. It remains to show that $\widetilde A^\ell_0$ is a subalgebra of $A$, that $\widetilde A^\ell_0 \cdot A^\ell_1 \subseteq A^\ell_1$ and that $A^\ell_1 \cdot A^\ell_1 \subseteq \widetilde A^\ell_0$.
    
    We already know that $\langle a_1, a_2, a_3 \rangle$ is a subalgebra of $A$. By \itemref{le:lm}{4}, we have $mn = s \in \widetilde A^\ell_0$. Further, we have $a_i m = a_i n = a_i + m + n \in \widetilde A^\ell_0$ for each $i$ by \itemref{le:lx}{3}, so we conclude that $\widetilde A^\ell_0$ is indeed a subalgebra of $A$.

    Next, we have
    \[ a_1 (b_1 + b_2) = (b_1 + b_2) + (c_1 + c_3) \in A_1^\ell , \]
    and similarly for all other products of any $a_i$ with any basis element of $A_1^\ell$. Moreover, $m (b_1 + b_2) = m b_1 + m b_2 = 0$ by \itemref{le:lx}{1} and
    \[ m (c_1 + c_2) = m c_1 + m c_2 = (\ell + m + c_1) + (\ell + m + c_2) = c_1 + c_2 \in A_1^\ell \]
    by \itemref{le:lx}{3}. Similarly, $n (b_1 + b_2) \in A_1^\ell$ and $n (c_1 + c_2) = 0$, so we conclude that indeed $\widetilde A^\ell_0 \cdot A^\ell_1 \subseteq A^\ell_1$.
    
    Finally, we compute that
    \begin{align*}
        (b_1 + b_2)(b_1 + b_3) &= b_1 b_3 + b_1 b_2 + b_2 b_3 = m + m + m = m \in \tilde A_0^\ell \ \text{ and} \\
        (b_1 + b_2)(c_1 + c_2) &= b_1 c_1 + b_1 c_2 + b_2 c_1 + b_2 c_2 \\
        &= (a_1 + b_1 + c_1) + (a_3 + b_1 + c_2) + (a_3 + b_2 + c_1) + (a_2 + b_2 + c_2) \\
        &= a_1 + a_2 \in \tilde A_0^\ell
    \end{align*}
    (and similarly for the other products) so we conclude that also $A^\ell_1 \cdot A^\ell_1 \subseteq \widetilde A^\ell_0$.
\end{proof}

For the reduced nilpotent Matsuo algebra, however, we again have a decomposition into (proper) eigenspaces, i.e., the operators $\ad_\ell$ are semisimple.

\begin{corollary}
    Let $\geom$ be the affine plane of order $3$ and let $A = M_R(\geom, 1)$. Then the reduced nilpotent Matsuo algebra $A' = A / \Ann(A)$ is an $8$-dimensional decomposition algebra with fusion law
	\begin{center}
		\renewcommand{\arraystretch}{1.2}
		\setlength{\tabcolsep}{0.75em}
		\begin{tabular}[b]{c|ccc}
			$*$	& $0$ & $1$ \\
			\hline
			$0$ & $\{ 0 \}$ & $\{ 1 \}$ \\
			$1$ & $\{ 1 \}$ & $\{ 0 \}$
		\end{tabular},
	\end{center}
    where for each of the $12$ lines $\ell \in \cL$, there is a decomposition $A' = A'^\ell_0 \oplus A'^\ell_1$ into eigenspaces for $\ad_\ell$ with eigenvalues $0$ and $1$, respectively.
\end{corollary}
\begin{proof}
    By \cref{th:aff-decalg}, we have, for each $\ell \in \cL$, a decomposition into generalized eigenspaces $A' = \widetilde A'^\ell_0 \oplus A'^\ell_1$.
    However, since $s = 0$ in $A'$, we have $\ell m = \ell n = 0$, so that $m$ and $n$ are eigenvectors with eigenvalue $0$ in $A'$ and hence $\widetilde A'^\ell_0 = A'^\ell_0$.
\end{proof}

\subsection{Fischer spaces of symplectic type}

We can now move on to Fischer spaces of higher rank.
We first deal with Fischer spaces of symplectic type (see \cref{def:symplectic}).
Let $\geom$ be a Fischer space of symplectic type and let $\ell$ be any line of~$\geom$.
Then any point of $p \not\in \ell$ is collinear with either $0$ or $2$ points of $\ell$.
(Indeed, as soon as a point $p$ is collinear with one point of $\ell$, the subspace spanned by $p$ and $\ell$ must be a complete quadrilateral.)

If $\pi \subset \cP$ is the set of points of a complete quadrilateral inside $\geom$, then we will denote the sum of the points of $\pi$ in the algebra $A$ also by $\pi$.

\begin{proposition}\label{pr:symplectic-dec}
    Let $\geom$ be a Fischer space of symplectic type and let $\ell = \{ a,b,c \}$ be any line of $\geom$.
    Let $\cP_0$ be the set of points of $\geom$ not on $\ell$ and not collinear with $a$, $b$ nor $c$ and let $\cP_2$ be the set of points collinear with exactly $2$ points of $\ell$.
    
    Then we have a decomposition $A = A^\ell_0 \oplus A^\ell_1$ into eigenspaces for the operator $\ad_\ell$ with eigenvalues $0$ and $1$, respectively, where
    \begin{align*}
        A^\ell_0 &= \langle a, b, c\rangle \oplus \langle \pi \mid \pi \text{ a complete quadrilateral with } \ell \subset \pi \rangle \oplus \langle w \mid w \in \cP_0 \rangle , \\
        A^\ell_1 &= \langle \ell z \mid z \in \cP_2 \rangle .
    \end{align*}
\end{proposition}
\begin{proof}
    Since $\ad_\ell$ is idempotent by \cref{le:lla=la}, it is semisimple with eigenvalues $0$ and $1$ only, so we have a decomposition $A = A^\ell_0 \oplus A^\ell_1$ into eigenspaces.
    Of course, $\langle a, b, c\rangle \oplus \langle w \mid w \in \cP_0 \rangle \leq A^\ell_0$.
    Moreover, by \cref{pr:des-decs}, we already know that for each complete quadrilateral $\pi$ containing $\ell$, we have
    \begin{align*}
        & \langle \pi \rangle \leq A^\ell_0 \quad \text{and} \\
        & \langle \ell z \mid z \in \pi \rangle \leq A^\ell_1 .
    \end{align*}
    For each choice of $\pi$, this gives us a $1$\dash dimensional contribution to $A_0^\ell$ and a $2$\dash dimensional contribution to $A_1^\ell$. Since $\cP$ is a basis for $A$, these contributions all intersect trivially and together with $\langle a, b, c\rangle$ and $\langle w \mid w \in \cP_0 \rangle$, they span all of $A$.
    The result follows.
\end{proof}
We will show in \cref{th:symplectic} below that these decompositions will again give rise to a $\Z/2\Z$-graded decomposition algebra. We need a number of auxiliary results first.
To begin, we need to understand the possible interactions between two different complete quadrilaterals through the line $\ell$.
\begin{proposition}\label{pr:2CQ}
    Let $\geom$ be a Fischer space of symplectic type and let $\ell = \{ a,b,c \}$ be a line of $\geom$.
    Suppose that $\ell$ is contained in two different complete quadrilaterals $\pi$ and $\pi'$, labeled as follows:
    
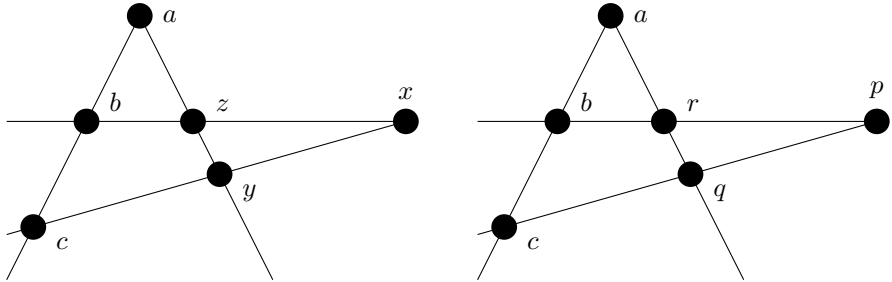
\[
    \begin{tikzpicture}[scale=.7]
        \DATwoDrawing{$a$}{$b$}{$c$}{$x$}{$y$}{$z$}
    \end{tikzpicture}
    \qquad
    \begin{tikzpicture}[scale=.7]
        \DATwoDrawing{$a$}{$b$}{$c$}{$p$}{$q$}{$r$}
    \end{tikzpicture}
    \]
    Then precisely two cases can occur:
    \begin{enumerate}[\rm (a)]
        \item\label{pr:2CQ:a} Either
        \begin{alignat*}{3}
            &p \sim x, &\quad & p \nsim y, &\quad & p \nsim z \\
            &q \nsim x, && q \sim y, && q \nsim z \\
            &r \nsim x, && r \nsim y, && r \sim z.
        \end{alignat*}
        In this case, $w := p \wedge x = q \wedge y = r \wedge z$. Moreover, $w \in \cP_0$, i.e., $w \nsim a$, $w \nsim b$ and $w \nsim c$.
        \item\label{pr:2CQ:b} Or
        \begin{alignat*}{3}
            &p \nsim x, &\quad & p \sim y, &\quad & p \sim z \\
            &q \sim x, && q \nsim y, && q \sim z \\
            &r \sim x, && r \sim y, && r \nsim z.
        \end{alignat*}
        In this case, we have
        \[ d := r \wedge y = q \wedge z, \quad e := r \wedge x = p \wedge z, \quad f := q \wedge x = p \wedge y . \]
        The points $a,b,c,d,e,f$ form another complete quadrilateral $\pi''$ through $\ell$, so in particular, we have lines $\{ a,e,f \}$, $\{ b,d,f \}$ and $\{ c,d,e \}$.
    \end{enumerate}
\end{proposition}
\begin{proof}
    We repeatedly use the fact that for any point and any line, there are either $0$ or $2$ points on that line collinear with that point.
    \begin{enumerate}[\rm (a)]
        \item 
        Assume first that $p \sim x$. Since also $p \sim b$, we must have $p \nsim z$. Similarly, since $p \sim c$, we have $p \nsim y$. On the other hand, we now have $y \sim c$ and $y \nsim p$, hence $y \sim q$. Continuing in this fashion, we find all $9$ relations between $p,q,r$ and $x,y,z$ as stated.
        
        Next, we look at the subspace spanned by the lines $\{ c, y, x \}$ and $\{ c, q, p \}$. This, too, is a complete quadrilateral, and since we already have $p \sim x$ and $q \sim y$, we must have $w := p \wedge x = q \wedge y$. Similarly, $w = r \wedge z$.
        
        Finally, we note that $w \sim x$ and $w \sim y$ imply $w \nsim c$, and similarly $w \nsim a$ and $w \nsim b$, so $w \in \cP_0$.
        \item 
        Assume now that $p \nsim x$. Since $p \sim b$, this implies $p \sim z$, and since $p \sim c$, this implies $p \sim y$. Now $y \sim p$ and $y \sim c$ imply $y \nsim q$. Again, we continue in this fashion to find all $9$ relations between $p,q,r$ and $x,y,z$ as stated.

        Next, we look at the subspace spanned by the lines $\{ a, z, y\}$ and $\{ a, r, q \}$. Again using the fact that this is a complete quadrilateral, the known collinearities between $y,z$ and $q,r$ tell us that $r \wedge y = q \wedge z$. Similarly, $r \wedge x = p \wedge z$ and $q \wedge x = p \wedge y$.
        
        Finally, we look at the subspace spanned by the line $\{ a, r, q \}$ and the point~$x$. Since $x \sim r$ and $x \sim q$, this is again a complete quadrilateral, containing $e = r \wedge x$ and $f = q \wedge x$. This forces $\{ a, e, f \}$ to be a line. Similarly, we get lines $\{ b, d, f \}$ and $\{ c, d, e \}$. It is now clear that the points $a,b,c,d,e,f$ form another complete quadrilateral.
        \qedhere
    \end{enumerate}
\end{proof}
\begin{remark}
    In principle, \cref{pr:2CQ} is nothing more than a geometric proof of the classification of symplectic Fischer spaces of rank $4$, which is, of course, a well known result (see also \cref{th:classif} below). The situation of \cref{pr:2CQ}\cref{pr:2CQ:a} corresponds to the Fischer space of $W(A_4)$, while the situation \cref{pr:2CQ:b} corresponds to the Fischer space of $W(D_4)$.
    However, our geometric version will be useful in the proof of \cref{th:symplectic} below.
\end{remark}

We also have a converse for this result for points in $\cP_0$.
\begin{proposition}\label{pr:converse P0}
    Let $\geom$ be a Fischer space of symplectic type, let $\ell = \{ a,b,c \}$ be a line of $\geom$ and let $\cP_0$ be the set of points not collinear to any of $a,b,c$.
    If $w \in \cP_0$, then there exist two complete quadrilaterals containing $\ell$ such that $w$ is as in \itemref{pr:2CQ}{a}.
\end{proposition}
\begin{proof}
    First, we recall that connected Fischer spaces have diameter at most $2$ by \cite{Hal93}*{(2.6)}, so for any point $w \in \cP_0$, there exists a complete quadrilateral $\pi$ labeled as in \cref{fig:CQ} such that $w \sim x$. Since $w \nsim b$ and $w \nsim c$, we must have $w \sim z$ and $w \sim y$ as well (again using the $0$-$2$-property).
    Now define $p := w \wedge x$, $q := w \wedge y$ and $r := w \wedge z$.
    Once again using the $0$-$2$-property several times, it follows that $\ell \cup \{ p,q,r \}$ forms a second complete quadrilateral $\pi'$.
    Now $\pi$, $\pi'$ and $w$ are indeed as in \itemref{pr:2CQ}{a}, as claimed.
\end{proof}

The points in $\cP_0$ form a disconnected union of subspaces of $\geom$, even if $\geom$ is not of symplectic type.
\begin{proposition}\label{pr:P0 subspace}
    Let $\geom$ be any Fischer space, let $\ell = \{ a,b,c \}$ be a line of $\geom$ and let $\cP_0$ be the set of points not collinear to any of $a,b,c$.
    If $v,w \in \cP_0$ with $v \sim w$, then also $v \wedge w \in \cP_0$.
\end{proposition}
\begin{proof}
    In any Fischer space, we have the $0$-$2$-$3$-property, i.e., given a point and a line of $\geom$, it is impossible to have exactly one point of that line collinear with that point. (Indeed, in a complete quadrilateral, we have the $0$-$2$-property, and in an affine plane, all $9$ points are collinear with each other.)
    
    Now assume that $v,w \in \cP_0$ with $v \sim w$ but with $u := v \wedge w \not\in \cP_0$. Then without loss of generality, we have $u \sim a$, but then $u$ is the only point of the line $\{ u,v,w \}$ collinear with $a$, a contradiction.
\end{proof}

We are now well prepared to show that our algebras are indeed decomposition algebras.
Unavoidably, the proof will require several case distinctions.
\begin{theorem}\label{th:symplectic}
    Let $\geom$ be a Fischer space of symplectic type. Then the nilpotent Matsuo algebra $A = M_R(\geom, 1)$ is a decomposition algebra with fusion law
	\begin{center}
		\renewcommand{\arraystretch}{1.2}
		\setlength{\tabcolsep}{0.75em}
		\begin{tabular}[b]{c|ccc}
			$*$	& $0$ & $1$ \\
			\hline
			$0$ & $\{ 0 \}$ & $\{ 1 \}$ \\
			$1$ & $\{ 1 \}$ & $\{ 0 \}$
		\end{tabular},
	\end{center}
    where for each line $\ell \in \cL$, there is a decomposition $A = A^\ell_0 \oplus A^\ell_1$ into eigenspaces for $\ad_\ell$ with eigenvalues $0$ and $1$, respectively.
\end{theorem}
\begin{proof}
    Let $\ell \in \cL$ and consider the decomposition $A = A^\ell_0 \oplus A^\ell_1$ into eigenspaces from \cref{pr:symplectic-dec}, with
    \begin{align*}
        A^\ell_0 &= \langle a, b, c\rangle \oplus \langle \pi \mid \pi \text{ a complete quadrilateral with } \ell \subset \pi \rangle \oplus \langle w \mid w \in \cP_0 \rangle , \\
        A^\ell_1 &= \langle \ell z \mid z \in \cP_2 \rangle .
    \end{align*}
    We first show that $A^\ell_0$ is a subalgebra of $A$. Notice that $\langle a,b,c \rangle$ is a subalgebra and that also $\langle w \mid w \in \cP_0 \rangle$ is a subalgebra by \cref{pr:P0 subspace}. For any $\pi$ through~$\ell$, we already know that $\langle a,b,c \rangle \cdot \pi \in A_0^\ell$ by \cref{th:des-decalg}, and obviously $\langle a,b,c \rangle \cdot w = 0$ for any $w \in \cP_0$. Next, we consider two distinct complete quadrilaterals $\pi, \pi'$ through~$\ell$. Labeling the points of $\pi$ and $\pi'$ as in \cref{pr:2CQ}, we have
    \begin{align*}
        \pi \cdot \pi'
        &= (\ell+x+y+z)(\ell+p+q+r) \\
        &= \ell (x+y+z) + \ell (p+q+r) + (x+y+z)(p+q+r) \\
        &= (x+y+z)(p+q+r) .
    \end{align*}
    If $\pi$ and $\pi'$ are as in \itemref{pr:2CQ}{a}, then
    \begin{align*}
        \pi \cdot \pi'
        &= (x+p+w)+(y+q+w)+(z+r+w) \\
        &= \pi + \pi' + w \in A_0^\ell.
    \end{align*}
    If $\pi$ and $\pi'$ are as in \itemref{pr:2CQ}{b}, then
    \begin{multline*}
        \pi \cdot \pi' = (x+q+f) + (x+r+e) + (y+p+f) + (y+r+d) \\ + (z+p+e) + (z+q+d) = 0.
    \end{multline*}
    Next, let $\pi$ be any complete quadrilateral through $\ell$, labeled as in \cref{fig:CQ}, and let $w \in \cP_0$. Of course, $\pi \cdot w = 0$ if $w$ is not collinear with any of $x,y,z$, so we may assume that $w \sim x$. By \cref{pr:converse P0} (and its proof), $\pi$ and $w$ are as in \itemref{pr:2CQ}{a}. Then
    \begin{align*}
        \pi \cdot w = (\ell + x + y + z)w
        &= (x + p + w) + (y + q + w) + (z + r + w) \\
        &= \pi + \pi' + w \in A_0^\ell.
    \end{align*}
    This covers all cases to show that $A_0^\ell$ is indeed a subalgebra of $A$.
    
    We now show that $A_0^\ell \cdot A_1^\ell \subseteq A_1^\ell$.
    Recall that if $\pi$ is a complete quadrilateral through $\ell$ labeled as in \cref{fig:CQ}, then $\ell z = a + b + x + y$.
    Moreover, we already know that $\langle a,b,c,\pi \rangle \cdot \ell x \in A_1^\ell$ by \cref{th:des-decalg}.
    Next, consider a complete quadrilateral $\pi'$ through $\ell$ different from $\pi$ and label it as in \cref{pr:2CQ}. Then
    \[ \pi' \cdot \ell z = \ell \cdot \ell z + (p + q + r)(a + b + x + y) . \]
    Notice that $\ell \cdot \ell z = \ell z$ by \cref{le:lla=la}. Moreover, we have
    \[ (p + q + r)(a + b) = (p+r+b) + (q+r+a) + (r+a+q) + (r+b+p) = 0 . \]
    If $\pi,\pi'$ are as in \itemref{pr:2CQ}{a}, then
    \begin{align*}
        (p + q + r)(x + y) &= (p+x+w) + (q+y+w) \\
        &= (a+b+x+y) + (a+b+p+q) = \ell z + \ell r ,
    \end{align*}
    so $\pi' \cdot \ell z = \ell r \in A_1^\ell$.
    If $\pi,\pi'$ are as in \itemref{pr:2CQ}{b}, then
    \begin{align*}
        (p + q + r)(x + y) &= (p+y+f) + (q+x+f) + (r+x+e) + (r+y+d) \\
        &= (a+b+d+e) + (a+b+p+q) = \ell f + \ell r ,
    \end{align*}
    so $\pi' \cdot \ell z = \ell z + \ell f + \ell r \in A_1^\ell$.
    Next, we let $w \in \cP_0$ and $z \in \cP_2$. If $w \nsim x$ and $w \nsim y$, then of course $w \cdot \ell z = 0$, so assume that $w \sim x$ (and hence also $w \sim y$ and $w \sim z$) so that again, $w$ and $\pi$ are as in \itemref{pr:2CQ}{a}. Then
    \begin{align*}
        w \cdot \ell z &= w (a + b + x + y) = (w+x+p) + (w+y+q) \\
        &= (a+b+x+y) + (a+b+p+q) = \ell z + \ell r \in A_1^\ell .
    \end{align*}
    
    Finally, we show that $A_1^\ell \cdot A_1^\ell \subseteq A_0^\ell$. Let $z, z' \in \cP_2$.
    If $z'$ belongs to the complete quadrilateral $\pi$ spanned by $\ell$ and $z$, then the result follows again from \cref{th:des-decalg}.
    Assume now that $\ell$ and $z'$ span a complete quadrilateral $\pi'$ different from $\pi$ and label it as in \cref{pr:2CQ}.
    We have to distinguish between $z' = r$ versus $z' = q$ (or $p$, but this amounts to interchanging $p$ and $q$ without loss of generality).
    
    So we first show that $\ell z \cdot \ell r \in A_0^\ell$. (In fact, we will have $\ell z \cdot \ell r = 0$.)
    Notice that $a \wedge b = p \wedge q = x \wedge y = c$, so by \itemref{le:lm}{2}, we have
    \begin{align*}
        (a+b)(x+y) &= a+b+x+y, \\
        (a+b)(p+q) &= a+b+p+q, \\
        (x+y)(p+q) &= x+y+p+q,
    \end{align*}
    hence
    \[ \ell z \cdot \ell r = (a+b+x+y)(a+b+p+q) = 0 . \]
    We now show that $\ell z \cdot \ell q \in A_0^\ell$.
    Using the fact that $(a+b)(a+c) = a+b+c$, that $(a+b)(p+r) = a+r+q$ and $(a+c)(x+y) = a+y+z$, we get
    \begin{align*}
        \ell z \cdot \ell q &= (a+b+x+y)(a+c+x+r) \\
        &= a+b+c+q+r+y+z + (x+y)(p+r) .
    \end{align*}
    If $\pi,\pi'$ are as in \itemref{pr:2CQ}{a}, then we get
    \[ \ell z \cdot \ell q = (a+b+c) + \pi + \pi' + w \in A_0^\ell . \]
    If $\pi,\pi'$ are as in \itemref{pr:2CQ}{b}, then we get
    \[ \ell z \cdot \ell q = \pi + \pi' + \pi'' + w \in A_0^\ell , \]
    where $\pi''$ is as in the statement of \itemref{pr:2CQ}{b}.
    We have now shown that $A_1^\ell \cdot A_1^\ell \subseteq A_0^\ell$ in all possible cases.
\end{proof}

\subsection{Fischer spaces not of symplectic type}

Finally, we show that when the Fischer space $\geom$ is not of symplectic type, then the corresponding nilpotent Matsuo algebra has a decomposition that is no longer $\Z/2\Z$-graded.
It suffices to go through the list of possible Fischer spaces of rank $4$ that are not of symplectic type.
It will turn out that in each case, the relation $A_1^\ell \cdot A_1^\ell \subseteq \tilde A_0^\ell$ fails for at least one line $\ell$.
(See, however, \cref{rem:some} below.)

Since this amounts to finding a single instance of a line and two elements of $A_1^\ell$ whose product does not lie in $\tilde A_0^\ell$, we have performed this task by computer, and we only present the result.

We first recall the classification of Fischer spaces of rank $4$.
\begin{theorem}\label{th:classif}
    Let $\geom = (\cP, \cL)$ be a Fischer space of rank $4$.
    Then $\geom$ is the Fischer space of one of the following $3$-transposition groups:
    \begin{enumerate}
        \item $W(A_4)$, the Weyl group of type $A_4$ (with $|\cP| = 10$), of symplectic type;
        \item $W(D_4)$, the Weyl group of type $D_4$ (with $|\cP| = 12$), of symplectic type;
        \item $3^3 \colon \Sym(4)$ (with $|\cP| = 18$), not of symplectic type;
        \item $2^6 \colon \mathrm{SU}_3(2)'$ (with $|\cP| = 36$), not of symplectic type;
        \item Marshall Hall's $3$-transposition group $[3^{10}] \colon 2$ (with $|\cP| = 81$) or its affine quotient $3^{3+3} \colon 2$ (with $|\cP| = 27$), both of ``Moufang type'', i.e., not containing any complete quadrilaterals but only affine planes.
    \end{enumerate}
\end{theorem}
\begin{proof}
    The proof of this fact, together with the size of $\cP$ in each case, can be found in \cite[Proposition (4.2)]{HS95}, where the authors point out that this classification has been proven independently by Zara, Hall and Moori; the first written source seems to be Zara's (unpublished) thesis from 1984.
\end{proof}

\begin{examples}\label{ex:rank4}
    \begin{enumerate}[(1)]
        \item Let $\geom$ be of type $3^3 \colon \Sym(4)$ and let $\ell = \{ a,b,c \}$ be a line contained in a complete quadrilateral labeled as in \cref{fig:CQ}. Let $m = \{ a', b', c' \}$ be a line parallel to $\ell$ in some affine plane through $\ell$, labeled such that the lines $aa'$, $bb'$ and $cc'$ are parallel. Then $a'+b' \in A_1^\ell$ and $\ell z \in A_1^\ell$, but their product is not contained in $\tilde A_0^\ell$.
        \item\label{ex:AG33} Let $\geom$ be the affine $3$-space, and label its $27$ points by $[p, q, r]$ with $p,q,r \in \F_3$. Let $\ell$ be the line $\{ [0,0,0], [1,0,0], [2,0,0] \}$. Then $[0,1,0] + [1,1,0] \in A_1^\ell$ and $[1,0,1] + [2,0,1] \in A_1^\ell$, but their product is not contained in $\tilde A_0^\ell$. (In fact, this product is again an element of $A_1^\ell$ itself.)
        \item Let $\geom$ be the Fischer space of type $2^6 \colon \mathrm{SU}_3(2)'$. Explicitly, this group is the semidirect product of a $3$-dimensional vector space $V$ over $\F_4 = \{ 0, 1, \omega, \omega+1 \}$ (with $\omega^2 = \omega+1$) with the group $\mathrm{SU}_3(2)'$ generated by the involutions
            \[ d=\begin{psmallmatrix} 1 & 0 & 0 \\ 1 & 1 & 0 \\ 1 & 0 & 1 \end{psmallmatrix} , \
                e=\begin{psmallmatrix} 1 & 1 & 0 \\ 0 & 1 & 0 \\ 0 & \omega+1 & 1 \end{psmallmatrix}, \
                f=\begin{psmallmatrix} 1 & 0 & 1 \\ 0 & 1 & \omega \\ 0 & 0 & 1 \end{psmallmatrix} .
            \]
            (See \cite[p.\@~288]{Hal93}.)
            The points of $\geom$ are now in one-to-one correspondence with the conjugates of these $3$ involutions inside this semidirect product, and we will represent them in the form $[v, g]$ where $v \in V$ and $g \in \langle d,e,f \rangle$, with product given by $[v, g][w, h] = [v \cdot h + w, gh]$.
            
            Now let $\ell$ be the line $\{ [0, d], [0, e], [0, ded] \}$.
            Then $[0, f] + [0, defed] \in A_1^\ell$ and $[(0,0,\omega), f] + [(\omega+1, 1, \omega), defed] \in A_1^\ell$, but their product is not contained in~$\tilde A_0^\ell$ (and again, this is, in fact, an element of $A_1^\ell$).
        \item Let $\geom$ be the Fischer space of type $[3^{10}] \colon 2$, with $81$ points. We can use a similar configuration as in \cref{ex:AG33}. Explicitly, label the $81$ points of $\geom$ by $[p,q,r,s]$ with $p,q,r,s \in \F_3$ as in \cite[Example 2.7]{DRS23} (which is based on \cite{LB83}). For $\ell$, we then take the line $\{ [0,0,0,0], [1,0,0,0], [2,0,0,0] \}$. Then $[0,1,0,0] + [1,1,0,0] \in A_1^\ell$ and $[0,0,0,1] + [1,0,0,1] \in A_1^\ell$, but their product is not contained in $\tilde A_0^\ell$ (and once again, this is an element of $A_1^\ell$).
    \end{enumerate}
\end{examples}
\begin{theorem}\label{th:main}  
    Let $\geom$ be a Fischer space, with corresponding nilpotent Matsuo algebra $A = M_R(\geom, 1)$.
    For each line $\ell$ of $\geom$, consider the decomposition into the generalized eigenspaces $\tilde A_0^\ell \oplus \tilde A_1^\ell$.
    
    Then $A$ is a decomposition algebra with fusion law
	\begin{center}
		\renewcommand{\arraystretch}{1.2}
		\setlength{\tabcolsep}{0.75em}
		\begin{tabular}[b]{c|ccc}
			$*$	& $0$ & $1$ \\
			\hline
			$0$ & $\{ 0 \}$ & $\{ 1 \}$ \\
			$1$ & $\{ 1 \}$ & $\{ 0 \}$
		\end{tabular},
	\end{center}
    if and only if $\geom$ is of symplectic type or $\geom$ is an affine plane.
\end{theorem}
\begin{proof}
    This now follows from \cref{th:aff-decalg,th:symplectic,ex:rank4}.
\end{proof}

\begin{remark}\label{rem:some}
    Interestingly, it can happen that for \emph{some} lines $\ell$, the corresponding decomposition still admits a $\Z/2\Z$-grading, but not for all lines. This happens, for instance, for the Fischer spaces of type $3^{n} \colon \Sym(n+1)$, for all $n \geq 3$. If we would restrict to only the ``good'' lines in such a Fischer space, then we would still get a decomposition algebra with a $\Z/2\Z$-grading, but with too few decompositions to be interesting. Notice that these ``good lines'' are precisely the ones that are only contained in affine planes (and not in any complete quadrilateral). The $3$ points on such a line are ``$\theta$-equivalent'' in the terminology of \cite[p.\@~150]{CH95}.
\end{remark}

\section{Miyamoto group for the complete quadrilateral}

The small example of the complete quadrilateral is particularly interesting because it admits a $\Z$-grading, and not just a $\Z/2\Z$-grading.
Indeed, recall from Theorem \ref{th:des-decalg} that the fusion law is given by
\begin{center}
	\renewcommand{\arraystretch}{1.2}
	\setlength{\tabcolsep}{0.75em}
	\begin{tabular}[b]{c|ccc}
		$*$	& $0$ & $1$ \\
		\hline
		$0$ & $\{ 0 \}$ & $\{ 1 \}$ \\
		$1$ & $\{ 1 \}$ & $\emptyset$
	\end{tabular}.
\end{center}
Notice that the map from $X = \{ 0, 1 \}$ to $\Z$ sending $0$ to $0$ and $1$ to $1$ determines a fusion law morphism (see \cref{def:grading}).
In particular, the algebras of Theorem \ref{th:des-decalg} and Corollary \ref{cor:des-decalg'} can be viewed as $\Z$-decomposition algebras.
Observe that the $R$-character group of $\Gamma = \Z$ is naturally isomorphic to $R^\times$.

\begin{definition} \label{def:MiyZ}
    Let $(A, \I, \Omega)$ be a $\Z$-decomposition algebra.
    For $i \in \I$ and $\lambda \in R^{\times}$, write $\tau_{i,\lambda}$ for the Miyamoto map given by
    \[ \tau_{i,\lambda}(a) = \lambda^{k} a \quad \text{for all } a \in A^i_k. \]
\end{definition}

\begin{remark}\label{rem:tau_ell}
	If $A = M_R(\geom, 1)$ and $\ell$ is a line of $\geom$ such that $A$ decomposes into eigenspaces $A = A_0^\ell \oplus A_1^\ell$, then we can write the corresponding Miyamoto map as
	\[ \tau_{\ell,\lambda} = \id + (1 + \lambda) \ad_\ell . \]
\end{remark}

Let $(A,\I,\Omega)$ be a $\Z$-decomposition algebra. For each fixed $i$, we have a multiplicative group
\[  T_{i}(R) := \langle \tau_{i, \lambda} \mid \lambda \in R^\times \rangle \cong R^\times \]
of automorphisms, so it makes sense to investigate the corresponding Miyamoto groups as \emph{algebraic groups}, rather than as abstract ones (at least when $A$ is finite-dimensional over a field).
Following \cite{Mil17}, we will view algebraic groups as functors from the category of commutative $k$-algebras to the category of groups.

\begin{definition}
Fix a base field $k$ and let $(A,\I,\Omega)$ be a $\Z$-decomposition algebra, with $A$ finite-dimensional over $k$. Define the \emph{(full) Miyamoto group} of $(A, \I, \Omega)$ as the algebraic group over $k$ generated by the tori $T_i$:
\[ \Miy(A, \I, \Omega) := \langle T_i \mid i \in \I \rangle \leq \Aut(A) . \]
\end{definition}

\begin{remark}\label{rem:Rpts1}
    Recall that, by definition, this is the smallest \emph{closed} subgroup of $\Aut(A)$ containing these tori $T_i$. In particular, it will \emph{not} be true in general that the group of $R$-points $\Miy(A, \I, \Omega)(R)$ is generated by the groups $T_i(R)$ for all $k$-algebras $R$. We will see this explicitly in \cref{rem:Rpts2} below.
\end{remark}

From now on, let $A$ be the $6$-dimensional algebra from \cref{th:des-decalg} and let $A' := A / \Ann(A)$ be its $5$-dimensional quotient as in \cref{cor:des-decalg'}. 
We investigate their Miyamoto groups as algebraic groups over $\F_2$.
We keep the notations from \cref{ss:CQ}, so in particular, the points $a,b,c,x,y,z$ are labeled as in \cref{fig:CQ}.
%

\begin{lemma}\label{le:CQ-Miy}
    Let $\ell = \{ a,b,c \}$ as in \cref{fig:CQ} and let $\lambda \in R^\times$. Then the Miyamoto maps $\tau_{\ell, \lambda}$ of $A$ fix $a$, $b$ and $c$ and
    \begin{align*}
        \tau_{\ell, \lambda}(x) &= x + (1 + \lambda) \ell x , \\
        \tau_{\ell, \lambda}(y) &= y + (1 + \lambda) \ell y , \\
        \tau_{\ell, \lambda}(z) &= z + (1 + \lambda) \ell z .
    \end{align*}
\end{lemma}
\begin{proof}
    This follows immediately from \cref{rem:tau_ell}.
\end{proof}
The previous lemma suggests to use the basis
\begin{equation}\label{eq:basis}
    \B = (a,\ b,\ \ell,\ \underbrace{b+c+y+z}_{\ell x},\ \underbrace{a+c+x+z}_{\ell y},\ s)
\end{equation}
for $A$, and that is what we will be doing.
In fact, the multiplication table of $A$ with respect to this basis $\B$ looks very simple:
\[
\label{ta:A}
\begin{array}{c|cccccc}
    \cdot & a & b & \ell & \ell x & \ell y & s \\
    \hline
    a & 0 & \ell & 0 & 0 & \ell x & 0 \\
    b & \ell & 0 & 0 & \ell y & 0 & 0 \\
    \ell & 0 & 0 & 0 & \ell x & \ell y & 0 \\
    \ell x & 0 & \ell y & \ell x & 0 & 0 & 0 \\
    \ell y & \ell x & 0 & \ell y & 0 & 0 & 0 \\
    s & 0 & 0 & 0 & 0 & 0 & 0
\end{array}
\]
\begin{remark}\label{rem:AAA}
    From the multiplication table, we see that $A \cdot A = \langle \ell, \ell x, \ell y \rangle$ and $A \cdot (A \cdot A) = \langle \ell x, \ell y \rangle$, which gives us two interesting subalgebras that are invariant under arbitrary automorphisms of~$A$.
    Recall that $\Ann(A) = \langle s \rangle$, which gives us a third subalgebra invariant under $\Aut(A)$.
\end{remark}

The following notation will turn out to be useful.
\begin{definition}
    \begin{enumerate}
        \item 
            For each $\alpha,\beta \in R$, we define a $3$ by $3$ matrix
            \[ M_{\alpha,\beta} := \begin{pmatrix}
                \alpha & 0 & \beta \\
                0 & \beta & \alpha \\
                0 & 0 & 0
            \end{pmatrix} . \]
        \item
            For each $3$ by $3$ matrix $M$ over $R$ and each $\lambda \in R^\times$, we define a $6$ by $6$ matrix
            \[ S_{M,\lambda} := \begin{pmatrix}
                I_3 & O_3 \\
                M & \begin{psmallmatrix} \lambda & & \\ & \lambda & \\ & & 1 \end{psmallmatrix}
            \end{pmatrix} , \]
            where $I_3$ is the $3$ by $3$ identity matrix and $O_3$ is the $3$ by $3$ zero matrix. Moreover, we define
            \[  S_{\alpha,\beta,\lambda} := S_{M_{\alpha,\beta},\lambda} \]
            for all $\alpha,\beta \in R$ and $\lambda \in R^\times$.
    \end{enumerate}
\end{definition}
\begin{lemma}\label{le:CQ-Miy1}
    Consider the algebraic group $G = \G_a^2 \rtimes \G_m$, where the (left) action of $\G_m$ on $\G_a^2$ is given by multiplication.
    Then the set of matrices
    \[ \{ S_{\alpha,\beta,\lambda} \mid \alpha,\beta \in R, \lambda \in R^\times \} \]
    forms a group under matrix multiplication, isomorphic to $G(R)$.
\end{lemma}
\begin{proof}
    It is straightforward to check that
    \[ S_{\alpha,\beta,\lambda} \cdot S_{\gamma,\delta,\mu} = S_{\alpha + \lambda\gamma,\, \beta + \lambda\delta,\, \lambda\mu} \]
    for all $\alpha,\beta,\gamma,\delta \in R$ and $\lambda,\mu \in R^\times$. In particular,
    \begin{align*}
        &\{ S_{\alpha,\beta,1} \mid \alpha,\beta \in R \} \cong R^2 , \\
        &\{ S_{0, 0, \lambda} \mid \lambda \in R^\times \} \cong R^\times ,
    \end{align*}
    and the conjugation action of the second subgroup on the first is given by
    \[ S_{0, 0, \lambda} \cdot S_{\alpha,\beta,1} \cdot S_{0, 0, \lambda}^{-1} = S_{\lambda\alpha,\lambda\beta,1} . \qedhere \]
\end{proof}
\begin{lemma}\label{le:CQ-Miy2}
    Consider the four lines $\ell_1 = \{ a,b,c \}$, $\ell_2 = \{ b,x,z \}$, $\ell_3 = \{ a, y, z \}$ and $\ell_4 = \{ c, x, y \}$.
    Then with respect to the basis $\B$ as in \cref{eq:basis}, the Miyamoto maps with respect to these four lines have matrix representations
    \begin{align*}
        \tau_{\ell_1,\lambda} &= S_{0,\, 0,\, \lambda} , \\
        \tau_{\ell_2,\lambda} &= S_{1+\lambda,\, 0,\, \lambda} , \\
        \tau_{\ell_3,\lambda} &= S_{0,\, 1+\lambda,\, \lambda} , \\
        \tau_{\ell_4,\lambda} &= S_{1+\lambda,\, 1+\lambda,\, \lambda} .
    \end{align*}
\end{lemma}
\begin{proof}
    This follows from \cref{le:CQ-Miy} and a little bit of computation.
\end{proof}

\begin{theorem}\label{th:CQ-Miy}
    We have $\Miy(A, \I, \Omega) \cong \Miy(A', \I, \Omega) \cong \G_a^2 \rtimes \G_m$, where the action of $\G_m$ on $\G_a^2$ is given by multiplication.
\end{theorem}
\begin{proof}
    Let $G = \G_a^2 \rtimes \G_m$. For each $k$-algebra $R$, let
    \[ H(R) = \{ S_{\alpha,\beta,\lambda} \mid \alpha,\beta \in R, \lambda \in R^\times \} \cong G(R) \]
    as in \cref{le:CQ-Miy1}.
    By \cref{le:CQ-Miy2}, we already have $\Miy(A, \I, \Omega) \leq H$, and to show equality, it only remains to show that the algebraic group $T$ generated by the four tori
    \begin{align*}
        T_1(R) &:= \langle S_{0,\, 0,\, \lambda} \mid \lambda \in R^\times \rangle, &
        T_2(R) &:= \langle S_{1+\lambda,\, 0,\, \lambda} \mid \lambda \in R^\times \rangle, \\
        T_3(R) &:= \langle S_{0,\, 1+\lambda,\, \lambda} \mid \lambda \in R^\times \rangle, &
        T_4(R) &:= \langle S_{1+\lambda,\, 1+\lambda,\, \lambda} \mid \lambda \in R^\times \rangle,
    \end{align*}
    coincides with the whole group $H$. 
    Observe that $S_{\alpha,\beta,\lambda} \cdot S_{0, 0, \lambda^{-1}} = S_{\alpha, \beta, 1}$ for all $\alpha,\beta \in R, \lambda \in R^\times$, so we already have $S_{1+\lambda, 0, 1} \in T(R)$ and $S_{0, 1+\lambda, 1} \in T(R)$ for all $\lambda \in R^\times$.
    For $R = k^s$, the separable closure of $k$, these matrices clearly span all matrices of the form $S_{\alpha,\beta,1}$, so together with the $S_{0,0,\lambda}$ for all $\lambda \in (k^s)^\times$, we get the whole group $\{ S_{\alpha,\beta,\lambda} \mid \alpha,\beta \in k^s, \lambda \in (k^s)^\times \} = H(k^s)$, so $T(k^s) = H(k^s)$.
    Since $T \leq H$ and $H \cong G$ is smooth, this implies that $T = H$ as algebraic groups.
    
    It remains to show that $\Miy(A, \I, \Omega) \cong \Miy(A', \I, \Omega)$. Notice that each Miyamoto map of $(A, \I, \Omega)$ fixes $\Ann(A)$ elementwise. Since $A' = A/\Ann(A)$, we get an epimorphism
    \[ \rho \colon \Miy(A, \I, \Omega) \twoheadrightarrow \Miy(A', \I, \Omega) . \]
    With respect to the basis $\B$ for $A$ as in \cref{eq:basis} and the corresponding basis
    \begin{equation}\label{eq:basis'}
        \B' = (a,\ b,\ \ell,\ \ell x,\ \ell y)
    \end{equation}
    for $A'$, this map is simply given by canceling the last row and the last column of the matrix.
    Since all matrices $S_{\alpha, \beta, \lambda}$ have the same last row and last column, this implies that $\rho$ is injective, and hence an isomorphism.
\end{proof}
\begin{remark}\label{rem:Rpts2}
    As indicated in \cref{rem:Rpts1}, it is not true that $\langle T_i(R) \rangle_{i=1,2,3,4} \cong (\G_a^2 \rtimes \G_m)(R)$ for all $k$-algebras $R$.
    For instance, take $R = k[x]$, the algebra of polynomials over $k$. Then $T_i(R) = R^\times = k^\times$, so that the group generated by the $T_i(R)$ is $k^2 \rtimes k^\times \not\cong R^2 \rtimes R^\times$.
\end{remark}
To finish this section, we point out that the full automorphism groups $\Aut(A)$ and $\Aut(A')$ are substantially larger than the Miyamoto groups.

\begin{proposition} \label{pr:DAG22-Aut}
    We have isomorphisms as algebraic groups
    \begin{enumerate}
        \item $\Aut(A') \cong \G_a^2 \rtimes \GL_2$ and
        \item $\Aut(A) \cong \G_a^2 \rtimes (\Aut(A') \times \G_m)$, where $\GL_2 \leq \Aut(A')$ acts ``quadratically'' and $\G_a^2 \leq \Aut(A')$ acts trivially on the first copy of $\G_a^2$, see below.
    \end{enumerate}
\end{proposition}
\begin{proof}
We will continue to use the bases $\B$ for $A$ and $\B'$ for $A'$ given in \eqref{eq:basis} and~\eqref{eq:basis'}.
It will be convenient to use the multiplication table from page~\pageref{ta:A}.
\begin{enumerate}
    \item 
        Let $\theta \in \Aut(A')$ be arbitrary. 
        By \cref{rem:AAA}, its matrix representation with respect to $\B'$ must be of the form
        \[ \begin{pmatrix}
        	\alpha_1 	& \beta_1 		& 0				& 0			& 0  	\\
        	\alpha_2	& \beta_2 		& 0				& 0			& 0 	\\
        	\alpha_3  	& \beta_3		& \gamma_3 		& 0			& 0 	\\
        	\alpha_4	& \beta_4		& \gamma_4  	& \delta_4	& \eta_4 \\
        	\alpha_5	& \beta_5		& \gamma_5		& \delta_5 	& \eta_5
        \end{pmatrix} .\]
        Since this matrix is invertible, so is its submatrix $\begin{psmallmatrix} \delta_4 & \eta_4 \\ \delta_5 & \eta_5 \end{psmallmatrix}$.
        The equality $\theta(\ell)\cdot\theta(\ell x) = \theta(\ell x) \neq 0$ then implies  $\gamma_3 = 1$.
        From $\theta(a) \theta(\ell x) = 0$ and $\theta(a) \theta(\ell y) = \theta(\ell x)$, we get
        \[ \begin{pmatrix}
        	\delta_4 & \delta_5 \\
        	\eta_4 	 & \eta_5 
        \end{pmatrix} 
        \cdot 
        \begin{pmatrix}
        \alpha_3\\
        \alpha_1
        \end{pmatrix}
        =
        \begin{pmatrix}
        	0 \\
        	\delta_4
        \end{pmatrix}
        \quad \text{and} \quad
        \begin{pmatrix}
        	\delta_4 & \delta_5 \\
        	\eta_4 	 & \eta_5 
        \end{pmatrix} 
        \cdot 
        \begin{pmatrix}
        \alpha_2\\
        \alpha_3
        \end{pmatrix}
        =
        \begin{pmatrix}
        	0 \\
        	\delta_5
        \end{pmatrix} .
        \]
        In particular, $\alpha_1, \alpha_2, \alpha_3$ are uniquely determined by $\delta_4,\delta_5,\eta_4,\eta_5$,
        and a similar argument holds for $\beta_1, \beta_2, \beta_3$.
        Next, using $\theta(a) \theta(\ell) = 0$, we get
        \[ \alpha_4 = \alpha_1 \gamma_5 + \alpha_3 \gamma_4 \quad \text{and} \quad \alpha_5 = \alpha_2 \gamma_4 + \alpha_3 \gamma_5 , \]
        so $\alpha_4$ and $\alpha_5$ are uniquely determined by the $\gamma$'s, $\delta$'s and $\epsilon$'s. The same holds for $\beta_4$ and $\beta_5$, so we conclude that the restriction morphism 
        \[ \rho \colon \Aut(A') \to \Aut(\langle \ell,\ell x, \ell y \rangle) \leq \GL_3 \]
        is injective.
        Moreover, by checking all remaining relations $\theta(e_1) \cdot \theta(e_2) = \theta(e_1 \cdot e_2)$, for $e_1,e_2 \in \{a,b,\ell,\ell x,\ell y\}$, we see that $\im(\rho) \leq \GL_3$ consists of \emph{all} invertible matrices of the form
        \[ \begin{pmatrix}
        	1 			& 		0 	& 0 		\\
        	\gamma_4   	& \delta_4	& \eta_4 	\\
        	\gamma_5	& \delta_5	& \eta_5
        \end{pmatrix}. \]
        It is now straightforward to check that $\im(\rho) \cong \G_a^2 \rtimes \GL_2$.
    \item
        Notice from the multiplication table that $A_1 = \langle a,b,\ell,\ell x,\ell y \rangle$ is a subalgebra of $A$, isomorphic to the quotient $A' = A / \langle s \rangle$. Moreover, each automorphism of $A_1$ extends to a unique automorphism of $A$ fixing $s$, so $\Aut(A')$ is isomorphic to a subgroup of $\Aut(A)$.
        Keeping in mind that the subalgebra $\langle \ell, \ell x, \ell y \rangle$ is stabilized by each automorphism, it follows that an arbitrary $\theta \in \Aut(A)$ has matrix form
        \[
        \left(\begin{array}{ccccc|c}
            & & & & & 0 \\
            & & & & & 0 \\
            & & \Theta & & & 0 \\
            & & & & & 0 \\
            & & & & & 0 \\
            \hline
            \kappa & \lambda & 0 & 0 & 0 & \nu 
        \end{array}\right)
        =
        \left(\begin{array}{ccccc|c}
        	\alpha_1 	& \beta_1 		& 0				& 0			& 0 & 0  	\\
        	\alpha_2	& \beta_2 		& 0				& 0			& 0 & 0 	\\
        	\alpha_3  	& \beta_3		& 1 		& 0			& 0 & 0 	\\
        	\alpha_4	& \beta_4		& \gamma_4  	& \delta_4	& \eta_4 & 0 \\
        	\alpha_5	& \beta_5		& \gamma_5		& \delta_5 	& \eta_5 & 0 \\
            \hline
            \kappa & \lambda & 0 & 0 & 0 & \nu 
        \end{array}\right)
        \]
        for some $\Theta \in \Aut(A_1)$ and some $\kappa, \lambda \in R$ and $\nu \in R^\times$.
        Hence $\Aut(A) \cong \G_a^2 \rtimes (\Aut(A_1) \times \G_m)$, where the action can be determined by multiplying two such matrices.
        Clearly, the $\G_m$ acts on $\G_a^2$ by multiplication.
        To determine the action of $\Aut(A_1)$ on $\G_a^2$, we see that the action is given by right multiplication with the upper left $(2 \times 2)$-block
        $\begin{psmallmatrix} \alpha_1 & \beta_1 \\ \alpha_2 & \beta_2 \end{psmallmatrix}$.
        Using the equations from part (i), we see that this block can be rewritten as
        \[
        \begin{pmatrix} \alpha_1 & \beta_1 \\ \alpha_2 & \beta_2 \end{pmatrix}
        = r \begin{pmatrix} \delta_4^2 & \eta_4^2 \\ \delta_5^2 & \eta_5^2 \end{pmatrix} ,
        \]
        where $r$ is precisely the inverse of the determinant of the matrix 
        $\begin{psmallmatrix} \delta_4 & \eta_4 \\ \delta_5 & \eta_5 \end{psmallmatrix}$,
        so that this block 
        $\begin{psmallmatrix} \alpha_1 & \beta_1 \\ \alpha_2 & \beta_2 \end{psmallmatrix}$
        has determinant $1$.
        In other words, the action of $\GL_2$ on $\G_a^2$ is given by first applying the group homomorphism
        \[ \GL_2 \to \SL_2 \colon M = \begin{psmallmatrix} \alpha & \beta \\ \gamma & \delta \end{psmallmatrix} \mapsto (\det M)^{-1} \begin{psmallmatrix} \alpha^2 & \beta^2 \\ \gamma^2 & \delta^2 \end{psmallmatrix} \]
        and then applying right multiplication on $\G_a^2$.
    \qedhere
\end{enumerate}
\end{proof}

\begin{remark}
    It is natural to ask whether we can use the $\Z/2\Z$-gradings of the other examples to produce interesting Miyamoto groups. The usual definition of the Miyamoto maps as in \cref{def:miy} seems less useful (for instance, they are trivial if the base ring $R$ has no nilpotent elements). A possible solution is to consider the automorphism group $\Aut(A)$ as an algebraic group in characteristic $2$ and define Miyamoto maps as elements of a copy of the algebraic group $\mu_2$, which is not smooth in characteristic $2$. The groups generated by these Miyamoto maps is then an algebraic subgroup of $\Aut(A)$.
    We have tried to perform this computation for the symplectic Fischer space associated to $\Sym(5)$, but we have not been able to find a meaningful description of the resulting Miyamoto group.
\end{remark}

\end{document}